\documentclass[12pt]{amsart}

\usepackage{lutz-kasfka-arx}

\begin{document}

	\title{Koszulness and supersolvability for Dirichlet arrangements}

\author{Bob Lutz}
\address{Department of Mathematics, University of Michigan, Ann Arbor, MI, USA}
\email{boblutz@umich.edu}
\thanks{Work of the author was partially supported by NSF grants DMS-1401224 and DMS-1701576.}

\date{\today}


\begin{abstract}
	We prove that the cone over a Dirichlet arrangement is supersolvable if and only if its Orlik-Solomon algebra is Koszul. This was previously shown for four other classes of arrangements. We exhibit an infinite family of cones over Dirichlet arrangements that are combinatorially distinct from these other four classes.
\end{abstract}

\maketitle

\section{Introduction}

A \emph{Koszul algebra} is a graded algebra that is ``as close to semisimple as it can possibly be'' \cite[p. 480]{beilinson1996}. Koszul algebras play an important role in the topology of complex hyperplane arrangements. For example, if $\AA$ is such an arrangement and $U$ its complement, then the Orlik-Solomon algebra $\os(\AA)$ is Koszul if and only if $U$ is a rational $K(\pi,1)$-space. Also if $\os(\AA)$ is Koszul and $G_1\triangleright G_2 \triangleright\cdots$ denotes the lower central series of the fundamental group $\pi_1(U)$, defined by $G_1=\pi_1(U)$ and $G_{n+1} = [G_n,G_1]$, then the celebrated \emph{Lower Central Series Formula} holds:
\begin{equation}
\prod_{k=1}^\infty (1-t^k)^{\phi_k} = P(U,-t),
\label{eq:lcsformula}
\end{equation}
where $P(U,t)$ is the Poincar\'{e} polynomial of $U$ and $\phi_k = \rk (G_k/G_{k+1})$.

It is natural to seek a combinatorial characterization of the arrangements $\AA$ for which $\os(\AA)$ is Koszul. Shelton and Yuzvinsky \cite[Theorem 4.6]{shelton1997} showed that if $\AA$ is supersolvable, then $\os(\AA)$ is Koszul. Whether the converse holds is unknown.

\begin{q}
	If the Orlik-Solomon algebra of a central hyperplane arrangement $\AA$ is Koszul, then is $\AA$ supersolvable?
	\label{q:koszul}
\end{q}

 We answer this question affirmatively for cones (or centralizations) over \emph{Dirichlet arrangements}, a generalization of graphic arrangements arising from electrical networks and order polytopes of finite posets \cite{lutz2017,lutz2018mat}.

\begin{thm}
	The cone over a Dirichlet arrangement is supersolvable if and only if its Orlik-Solomon algebra is Koszul.
	\label{thm:intro}
\end{thm}

Question \ref{q:koszul} has been answered affirmatively for other classes of arrangements, including graphic arrangements \cite{hultman2016,jambu1998, schenck2002, vanle2013}. Our next theorem shows that Theorem \ref{thm:intro} properly extends all previous results. We say that two central arrangements are \emph{combinatorially equivalent} if the underlying matroids are isomorphic.

\begin{thm}
	There are infinitely many cones over Dirichlet arrangements that are not combinatorially equivalent to any arrangement for which Question \ref{q:koszul} has been previously answered.
	\label{thm:intro2}
\end{thm}

Dirichlet arrangements have also been called \emph{$\psi$-graphical arrangements} \cite{mu2015, stanley2015, suyama2018}. It was conjectured in \cite{mu2015} and proven in \cite{suyama2018} that the cone over a Dirichlet arrangement is supersolvable if and only if it is free (see also \cite{lutz2017}).

\section{Background}
\label{sec:bg}

\subsection{Dirichlet arrangements and supersolvability}

Let $\g=(V,E)$ be a finite connected undirected graph with no loops or multiple edges. Let $\B\subseteq V$ be a set of $\geq 2$ vertices inducing an edgless subgraph. We refer to the elements of $\B$ as \emph{boundary nodes}. Let $\be\subseteq E$ be the set of edges meeting $\B$. Let $\KK$ be a field of characteristic 0, and let $\u:\B\to \KK$ be injective.

\begin{mydef}
	The \emph{Dirichlet arrangement} $\overline{\AA}(\g,\u)$ is the arrangement in $\KK^{V\setminus \B}$ of hyperplanes given by
\begin{equation}
\{x_i = x_j:ij\in E\setminus \be\}\cup\{ x_i = \u(j):ij\in \be\mbox{ with }j\in \B\}.
\label{eq:intarr}
\end{equation}
\label{mydef:arr}
\end{mydef}

\begin{eg}[Wheatstone bridge]
	Consider the graph $\g$ on the left side of Figure \ref{fig:dir} with $V=\{i_1,i_2,j_1,j_2\}$, where the boundary nodes $j_1$ and $j_2$ are marked by white circles. Set $\KK=\R$, and let $ u(j_1)=1$ and $ u(j_2)=-1$. The Dirichlet arrangement $\overline{\AA}(\g,u)$ consists of the 5 hyperplanes $x_{i_1}=x_{i_2}$, $x_{i_1} = \pm1$ and $x_{i_2}=\pm 1$. This arrangement is illustrated on the right side of Figure \ref{fig:dir}.  
	
	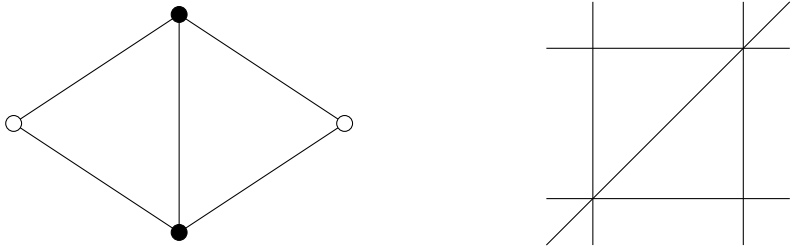
\begin{figure}[ht]
		\centering
		\begin{tikzpicture}[xscale=1,yscale=1]
\coordinate (a) at (-2.2,0);
\coordinate (b) at (0,1.45);
\coordinate (c) at (0,-1.45);
\coordinate (d) at (2.2,0);
\draw (a) -- (b) -- (d) -- (c) -- (a);
\draw (b) -- (c);
\draw[fill=white] (a) circle (3pt);
\draw[fill=black] (b) circle (3pt);
\draw[fill=black] (c) circle (3pt);
\draw[fill=white] (d) circle (3pt);

\def\a{6.5}
\draw ({\a-1},-1.618) -- ({\a-1},1.618);
\draw ({\a+1},-1.618) -- ({\a+1},1.618);
\draw ({\a-1.618},-1) -- ({\a+1.618},-1);
\draw ({\a-1.618},1) -- ({\a+1.618},1);
\draw ({\a-1.618},-1.618) -- ({\a+1.618},1.618);
\end{tikzpicture}
		\caption{A graph with boundary nodes marked in white, left, and a corresponding Dirichlet arrangement, right.}
		\label{fig:dir}
	\end{figure}
	
	\label{eg:dir}
\end{eg}

The arrangement $\overline{\AA}(\g,\u)$ is not central, i.e., the intersection of its elements is empty. We prefer to work with a centralized version of $\AA(\g,\u)$ with essentially the same combinatorics. If $\AA$ is an arrangement in $\KK^n$ defined by equations $f_i(x)=\alpha_i$ for homogenous functions $f_i$ and scalars $\alpha_i$, then the \emph{cone} over $\AA$ is the arrangement in $\KK^{n+1}$ defined by $f_i(x) = \alpha_ix_0$ for all $i$ and $x_0 = 0$, where $x_0$ is a new variable. The cone over any arrangement is central.

\begin{mydef}
	Let $\AA(\g,\u)$ denote the cone over the Dirichlet arrangement $\overline{\AA}(\g,\u)$.
\end{mydef}

Recall that the \emph{intersection lattice} of a central arrangement $\AA$ is the geometric lattice $L(\AA)$ of intersections of elements of $\AA$, ordered by reverse inclusion and graded by codimension.

\begin{mydef}
	A central arrangement $\AA$ is \emph{supersolvable} if 
	the intersection lattice $L(\AA)$ admits a maximal chain of elements $X$ satisfying
	\[\rk(X)+\rk(Y) = \rk(X\wedge Y)+\rk(X\vee Y)\]
	for every $Y\in L(\AA)$.
\end{mydef}

The graph $\g$ is \emph{chordal} if for any cycle $Z$ of length $\geq 4$ there is an edge of $\g\setminus Z$ with both endpoints in $Z$. Stanley \cite[Proposition 2.8]{stanley1972} proved that the graphic arrangement $\AA(\g)$ is supersolvable if and only if $\g$ is chordal. We have the following generalization for Dirichlet arrangements.

\begin{prop}[{\cite[Theorem 1.2]{lutz2017}}]
	Let $\mg$ be the graph obtained from $\g$ by adding edges between every pair of boundary nodes. The arrangement $\AA(\g,\u)$ is supersolvable if and only if $\mg$ is chordal. 
	\label{prop:2conn}
\end{prop}

\subsection{Orlik-Solomon algebras}

Given an ordered central arrangement $\AA$ over $\KK$, let $V$ be the $\KK$-vector space with basis $\{e_a : a\in \AA\}$. Let $\ea=\ea(V)$ be the exterior algebra of $V$. Write $xy=x\wedge y$ in $\ea$. The algebra $\ea$ is graded by taking $\ea^0 = \KK$ and $\ea^p$ to be spanned by all elements of the form $e_{a_1}\cdots e_{a_p}$.

Let $\dif:\ea\to\ea$ be the linear map defined by $\dif 1 = 0$, $\dif e_a = 1$ for all $a\in \AA$, and
\[\dif(xy) = \dif(x)y+(-1)^p x\dif(y)\]
for all $x\in \ea^p$ and $y\in \ea$.

The set $X$ is \emph{dependent} if the normal vectors of the hyperplanes in $X$ are linearly dependent. A \emph{circuit} is a minimal dependent set. If $X=\{a_1,\ldots,a_p\}\subseteq \AA$, assuming the $a_i$ are in increasing order, write $e_X = e_{a_1}\cdots e_{a_p}$ in $\ea$.

\begin{mydef}
	The \emph{Orlik-Solomon algebra} $\os(\AA)$ of a central arrangement $\AA$ is the quotient of $\ea$ by the \emph{Orlik-Solomon ideal}
	\begin{equation}
	I=\langle \dif(e_C) : C\subseteq \AA \mbox{ is a circuit}\rangle.
	\end{equation}
	That is, $\os(\AA) = \ea/I$.
\end{mydef}

 \subsection{Koszul algebras}
 We include the following definition of a Koszul algebra for completeness. A more thorough definition and further discussion can be found in \cite{peeva2010} and \cite{froberg1999}, respectively.
 
 \begin{mydef}
 	A graded $\KK$-algebra $A$ is \emph{Koszul} if the minimal free graded resolution of $\KK$ over $A$ is linear.
 \end{mydef}

  Quadraticity is a key property of Koszul algebras. A \emph{minimal generator} of the Orlik-Solomon algebra $I$ is an element of the form $\dif(e_C)$, where $C$ is a circuit and
\[\dif(e_C)\notin \langle \dif(e_X) : X\subseteq \AA\mbox{ is a circuit with }|X|<|C|\rangle.\]
If the minimal generators of $I$ are of degree 2, then $\os(\AA)$ is called \emph{quadratic}.

\begin{prop}[{\cite[Definition-Theorem 1]{froberg1999}}]
	If $\os(\AA)$ is Koszul, then $\os(\AA)$ is quadratic.
	\label{prop:kosquad}
\end{prop}
\section{Proof of Theorem \ref{thm:intro}}
\label{sec:pf}

We prove the following theorem, which implies Theorem \ref{thm:intro}.

\begin{thm}
	Let $\mg$ be the graph obtained from $\g$ by adding an edge between each pair of boundary nodes. The following are equivalent:
	\begin{enumerate}[(i)]
		\item $\mg$ is chordal
		\item $\AA(\g,\u)$ is supersolvable
		\item $\os(\AA(\g,\u))$ is Koszul
		\item $\os(\AA(\g,\u))$ is quadratic.
	\end{enumerate}
	\label{thm:koszul}
\end{thm}

We write $x$ instead of $\{x\}$ for all single-element sets. Let $\eh$ be an element not in $E$, and let $\eo=E\cup \eh$, so that $\AA(\g,\u)$ is indexed by $\eo$. Fix an ordering of $\eo$ with $\eh$ minimal. We say that $C\subseteq\eo$ is a \emph{circuit} if the corresponding subset of $\AA$ is a circuit.

\begin{mydef}
	A set $X\subseteq E$ is a \emph{crossing} if it is a minimal path between 2 distinct boundary nodes.
\end{mydef}

\begin{prop}[{\cite[Proposition 4.10]{lutz2018mat}}]
	A set $C\subseteq \eo$ is a circuit if and only if one of the following holds:
	\begin{enumerate}[(A)]
		\item $C = X\cup \eh$ for some crossing $X$
		\item $C\subseteq E$ is a cycle of $\g$ meeting at most 1 boundary node
		\item $C\subseteq E$ is a minimal acyclic set containing 2 distinct crossings.
	\end{enumerate}
	\label{prop:circuits}
\end{prop}

The circuits of type (C) in Proposition \ref{prop:circuits} come in two flavors: one contains 3 distinct crossings, while the other contains only 2. These are illustrated in Figure \ref{fig:circuits}. Circuits of type (C) containing only 2 distinct crossings are either disconnected, as pictured, or connected with both crossings meeting at a single boundary node.

\begin{figure}[ht]
		\begin{tikzpicture}[scale=1.5]
	\coordinate (o) at (0,0);
	\coordinate (a1) at (90:1/3);
	\coordinate (a2) at (90:2/3);
	\coordinate (a3) at (90:1);
	\coordinate (b1) at (210:1/2);
	\coordinate (b2) at (210:1);
	\coordinate (c1) at (330:1/3);
	\coordinate (c2) at (330:2/3);
	\coordinate (c3) at (330:1);
	\draw (o) -- (a1) -- (a2) -- (a3);
	\draw (o) -- (b1) -- (b2);
	\draw (o) -- (c1) -- (c2) -- (c3);
	\draw[fill=black] (o) circle (2pt);
	\draw[fill=black] (a2) circle (2pt);
	\draw[fill=black] (a1) circle (2pt);
	\draw[fill=white] (a3) circle (2pt);
	\draw[fill=white] (b2) circle (2pt);
	\draw[fill=black] (b1) circle (2pt);
	\draw[fill=black] (c2) circle (2pt);
	\draw[fill=black] (c1) circle (2pt);
	\draw[fill=white] (c3) circle (2pt);
	
	\def\a{3.5}
	\coordinate (d1) at (\a+1.5,-1/2);
	\coordinate (d2) at (\a+1.3,1/4);
	\coordinate (d3) at (\a+1.5,1);
	\coordinate (e1) at (\a,-1/2);
	\coordinate (e2) at (\a+0.2,0);
	\coordinate (e3) at (\a+0.2,1/2);
	\coordinate (e4) at (\a,1);
	\draw (d1) -- (d2) -- (d3);
	\draw (e1) -- (e2) -- (e3) -- (e4);
	\draw[fill=white] (d1) circle (2pt);
	\draw[fill=black] (d2) circle (2pt);
	\draw[fill=white] (d3) circle (2pt);
	\draw[fill=white] (e1) circle (2pt);
	\draw[fill=black] (e2) circle (2pt);
	\draw[fill=black] (e3) circle (2pt);
	\draw[fill=white] (e4) circle (2pt);
	\end{tikzpicture}
	\caption{Two circuits of $\eo$ with boundary nodes marked in white.}
	\label{fig:circuits}
\end{figure}
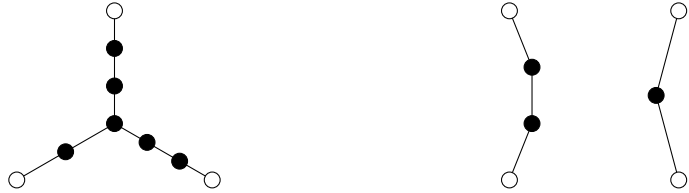

Taken together, the following 2 lemmas imply that circuits of type (C) do not contribute minimal generators to the Orlik-Solomon ideal $I$. When the usage is clear we will write $S = e_S$, so that $S$ is considered as an element of $\ea$ and a subset of $\eo$.

\begin{lem}
	Let $C\subseteq E$ be a circuit containing distinct crossings $X_1$, $X_2$ and $X_3$. In $\ea$ we have
	\[\dif(C) \in \langle \dif(\eh X_1),\dif(\eh X_2),\dif(\eh X_3)\rangle.\]
	
	\begin{proof}
		There are mutually disjoint paths $P_1, P_2,P_3\subseteq E$ in $\g$ such that $C=P_1\cup P_2\cup P_3$ and $X_i = P_j\cup P_k$ for distinct $i,j,k$. Write $a_i = |P_i|$, and suppose without loss of generality that $X_1 = P_2P_3$, $X_2 = P_1P_3$ and $X_3 = P_1P_2$ in $\ea$. We have
		\begin{align*}
		\dif(\eh X_3) &= P_1P_2 - \eh\dif(P_1)P_2 - (-1)^{a_1}\eh P_1\dif(P_2)\\
		\dif(\eh X_2) &= P_1P_3 - \eh\dif(P_1)P_3 - (-1)^{a_1}\eh P_1\dif(P_3)\\
		\dif(\eh X_1) &= P_2P_3 - \eh\dif(P_2)P_3 - (-1)^{a_2}\eh P_2\dif(P_3)
		\end{align*}
		Thus
		\begin{align*}
		\dif(P_3)\dif(\eh X_3) &= (-1)^{(a_1+a_2)(a_3-1)}(P_1P_2\dif(P_3) - \eh\dif(P_1)P_2\dif(P_3)\\ &\qquad- (-1)^{a_1}\eh P_1\dif(P_2)\dif(P_3))\\
		\dif(P_2)\dif(\eh X_2) &= (-1)^{a_1(a_2-1)}(P_1\dif(P_2)P_3 - \eh\dif(P_1)\dif(P_2)P_3 \\ &\qquad + (-1)^{a_1+a_2}\eh P_1\dif(P_2)\dif(P_3))\\
		\dif(P_1)\dif(\eh X_1) &= \dif(P_1)P_2P_3 +(-1)^{a_1}\eh\dif(P_1)\dif(P_2)P_3\\ &\qquad +(-1)^{a_1+a_2}\eh\dif(P_1) P_2\dif(P_3)
		\end{align*}
		Since $C=P_1P_2P_3$, we have
		\[\dif(C)=\dif(P_1)P_2P_3 + (-1)^{a_1}P_1\dif(P_2)P_3 + (-1)^{a_1+a_2}P_1P_2\dif(P_3),\]
		A computation now gives
		\[\dif(C)=\dif(P_1)\dif(\eh X_1)+(-1)^{a_1a_2}\dif(P_2)\dif(\eh X_2) + (-1)^{(a_1+a_2)a_3}\dif(P_3)\dif(\eh X_3),\]
		proving the result.
	\end{proof}
	\label{lem:osy}
\end{lem}

\begin{lem}
	Suppose that $X_1$ and $X_2$ are crossings such that no vertex in $V\setminus \B$ is met by both $X_1$ and $X_2$. In $\ea$ we have
	\[\dif(C)\in \langle \dif(\eh X_1),\dif(\eh X_2)\rangle.\]
	\begin{proof}
		The proof is similar to that of Lemma \ref{lem:osy}. In particular, we have
		\[\dif(X_1X_2) = \dif(X_1\eh)\dif(X_2) + \dif(X_1)\dif(X_2\eh),\]
		proving the result.
	\end{proof}
	\label{lem:os2cross}
\end{lem}

Let $C\subseteq \eo$ be a circuit. An element $i\in \eo$ is a \emph{chord} of $C$ if there exist circuits $C_1$ and $C_2$ such that $i = C_1\cap C_2$ and $C=(C_1\setminus C_2)\cup (C_2\setminus C_1)$. If $C$ admits no chord, then $C$ is \emph{chordless}.

\begin{prop}
	The minimal generators of $I$ are the elements of the form $\dif(C)$, where $C\subseteq \eo$ is a chordless circuit of type (A) or (B) in Proposition \ref{prop:circuits}.
	
	\begin{proof}
		Let $J$ be the ideal of $\ea$ generated by the elements of the form $\dif(C)$ for all circuits $C$ of types (A) and (B) in Proposition \ref{prop:circuits}. Note that any circuit of type (C) is described by either Lemma \ref{lem:osy} or \ref{lem:os2cross}. It follows that $J=I$ is the Orlik-Solomon ideal.
		
		Let $C\subseteq \eo$ be a circuit of type (A) or (B). It remains to show that $\dif(C)$ is a minimal generator of $I$ if and only if $C$ is chordless. Notice that a chord of $C$ is any edge $i\in E$ connecting two vertices met by $E\cap C$.
		
		Suppose first that $C$ is of type (B), and write $C=\{e_1,\ldots,e_r\}$. We have
		\[\dif(C) = \sum_{j=1}^r (-1)^{j-1} e_1\cdots \widehat{e_j}\cdots e_r.\]
		There is a chord $i$ of $C$ if and only if there is a circuit $C'$ of with a term of $\dif(C')$ dividing $e_2\cdots e_r$. Suppose that such a chord $i$ exists, and partition $C$ into two paths $P_1$ and $P_2$ such that $P_1\cup i$ and $P_2\cup i$ are cycles of $\g$. Write $a_j = |P_j|$, and suppose without loss of generality that $C=P_1P_2$ in $\ea$. We have
		\[\dif(C)=\dif(P_1)\dif(iP_2)+(-1)^{a_1a_2}\dif(P_2)\dif(iP_1),\]
		so $\dif(C)$ is not a minimal generator. Thus if $C$ is a cycle of $\g$, then $\dif(C)$ is a minimal generator of $I$ if and only if $C$ is chordless.
		
		Now suppose that $C =X\cup \eh$ for some crossing $X$. We have $\dif(C) = X - \eh\dif(X)$. There is a circuit $C'$ with a term of $\dif(C')$ dividing $X$ if and only if there is a chord $i$ of $X$. Suppose that such a chord $i$ exists. Partition $X$ into two sets $X_1$ and $X_2$ such that $X_1\cup i$ is a cycle of $\g$ and $X_2\cup i$ is a crossing. Write $b_j = |X_j|$, and suppose without loss of generality that $X=X_1X_2$ in $\ea$. We have
		\[(-1)^{b_1}\dif(C) = \dif(X_1)\dif(\eh i X_2) + (\eh \dif(X_2) + (-1)^{b_2}X_2)\dif(i X_1),\]
		where $X_1\cup i$ and $X_2\cup \{\eh, i\}$ are circuits of smaller size than $C$. Hence $\dif(C)$ is not a minimal generator. Thus if $C=X\cup \eh$ for some crossing $X$, then $\dif(C)$ is a minimal generator of $I$ if and only if $C$ is chordless. The result follows.
	\end{proof}
	\label{prop:mingen}
\end{prop}

\begin{prop}
	The graph $\mg$ is chordal if and only if there are no chordless circuits of type (A) or (B) in Proposition \ref{prop:circuits} having size $\geq 4$.
	
	\begin{proof}
		Let $\me$ be the set of edges of $\mg$ not in $E$. Suppose that $C$ is a chordless circuit of size $k\geq 4$. If $C=X\cup \eh$ is of type (A) for some crossing $X$, then there is $e\in \me$ such that $X\cup e$ is a cycle of $\mg$ admitting no chord. If $C$ is of type (B), then $C$ is a cycle of $\g$ (and hence $\mg$) admitting no chord. The ``only if'' direction follows. Now suppose that $\mg$ has a cycle $Z$ of size $\geq 4$ admitting no chord. Then either $Z\subseteq E$, in which case $Z$ is a circuit of type (B); or $Z\cap \me$ consists of a single edge $e$, in which case $(Z\setminus e)\cup \eh$ is a circuit of type (A).
	\end{proof}
	\label{prop:chordless}
\end{prop}

\begin{proof}[Proof of Theorem \ref{thm:koszul}]
	(i) $\Rightarrow$ (ii): This follows from Proposition \ref{prop:2conn}.
	(ii) $\Rightarrow$ (iv): This follows from \cite[Theorem 4.6]{shelton1997}. (iii) $\Rightarrow$ (iv): This is the content of Proposition \ref{prop:kosquad}. (iv) $\Rightarrow$ (i): This follows from Propositions \ref{prop:mingen} and \ref{prop:chordless}.
\end{proof}
\section{An infinite family}
\label{sec:eg}

 We prove Theorem \ref{thm:eg} below, which implies Theorem \ref{thm:intro2}. There are four classes of arrangements for which Question \ref{q:koszul} was previously answered:
\begin{enumerate}[(i)]
	\item Graphic arrangements
	\item Ideal arrangements
	\item Hypersolvable arrangements
	\item Ordered arrangements with disjoint minimal broken circuits.
\end{enumerate}
See \cite{hultman2016,jambu1998, schenck2002, vanle2013} for individual treatments. A priori it is unclear how these classes overlap with cones over Dirichlet arrangements.

Given a central arrangement $\AA$, let $M(\AA)$ be the usual matroid on $\AA$, so $X$ is independent in $M(\AA)$ if and only if the set of normal vectors of $X$ is linearly independent. For more on matroids and central arrangements, see \cite{stanley2007}.

\begin{mydef}
	Two central arrangements are \emph{combinatorially equivalent} if their underlying matroids are isomorphic.
\end{mydef}

\begin{mydef}
	Let $\chi(\g,\B)$ denote the chromatic number of the graph with vertex set $\B$ and an edge between $i$ and $j$ if and only if there is a crossing in $\g$ connecting $i$ and $j$.
	\label{mydef:omega}
\end{mydef}

\begin{eg}
	Consider the graph $\g$ on the left side of Figure \ref{fig:chrom} with $\B$ marked in white. On the right side is the graph with vertex set $\B$ and an edge between $i$ and $j$ if and only if there is a crossing in $\g$ connecting $i$ and $j$. This graph can be colored using 6 colors, as pictured, and no fewer, since it contains a clique on 6 vertices. Hence $\chi(\g,\B)=6$.

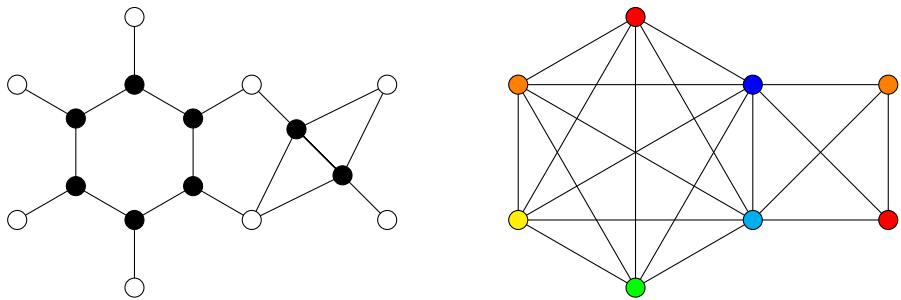
\begin{figure}
	\centering
	\begin{tikzpicture}[scale=1.8]
\def\c{0.5}
\foreach \s in {1,...,6}
{
	\coordinate (a\s) at ({60*\s+30}:1);
	\coordinate (b\s) at ({60*\s+30}:\c);
	\draw (a\s) -- (b\s);
};
\draw (b1) -- (b2) -- (b3) -- (b4) -- (b5) -- (b6) -- (b1);

\coordinate (a7) at ($(a5) + (1,0)$);
\coordinate (a8) at ($(a6) + (1,0)$);
\coordinate (b7) at ($0.33*(a6)+0.67*(a7)$);
\coordinate (b8) at ($0.67*(a6)+0.33*(a7)$);
\draw (a6) -- (b7) -- (b8) -- (a7);
\draw (b7) -- (a8) -- (b8) -- (a5) -- (b7);

\foreach \s in {1,...,8}
{
	\draw[fill=white] (a\s) circle (2pt);
	\draw[fill=black] (b\s) circle (2pt);
};

\foreach \s in {1,...,8}
{
	\coordinate (c\s) at ($(a\s)+(3.7,0)$);
};
\draw (c1) -- (c2);
\draw (c1) -- (c3);
\draw (c1) -- (c4);
\draw (c1) -- (c5);
\draw (c1) -- (c6);
\draw (c2) -- (c3);
\draw (c2) -- (c4);
\draw (c2) -- (c5);
\draw (c2) -- (c6);
\draw (c3) -- (c4);
\draw (c3) -- (c5);
\draw (c3) -- (c6);
\draw (c4) -- (c5);
\draw (c4) -- (c6);
\draw (c5) -- (c6);
\draw (c5) -- (c7);
\draw (c5) -- (c8);
\draw (c6) -- (c7);
\draw (c6) -- (c8);
\draw (c7) -- (c8);
\draw[fill=red] (c1) circle (2pt);
\draw[fill=orange] (c2) circle (2pt);
\draw[fill=yellow] (c3) circle (2pt);
\draw[fill=green] (c4) circle (2pt);
\draw[fill=cyan] (c5) circle (2pt);
\draw[fill=blue] (c6) circle (2pt);
\draw[fill=red] (c7) circle (2pt);
\draw[fill=orange] (c8) circle (2pt);
\end{tikzpicture}
	\caption{A graph with boundary nodes marked in white and an illustration of the associated number $\chi(\g,\B)$.}
	\label{fig:chrom}
\end{figure}
\end{eg}

\begin{thm}
	Suppose that $|E|\geq 240$ and $\chi(\g,\B)\geq 4$, and that some vertex of $\g$ is adjacent to at least 3 boundary nodes. If $\g\setminus \B$ contains the wheel graph on 5 vertices as an induced subgraph, then $\AA(\g,\u)$ is not combinatorially equivalent to any arrangement for which Question \ref{q:koszul} was previously answered.
	\label{thm:eg}
\end{thm}

\begin{eg}
	Recall that the \emph{join} $G+H$ of 2 graphs $G$ and $H$ is the disjoint union of $G$ and $H$ with edges added between every vertex of $G$ and every vertex of $H$. The join of any finite number of  graphs is defined by induction. Let $\overline{K}_n$ and $K_n$ be the edgeless and complete graphs, resp., on $n$ vertices. Let $W_5$ be the wheel graph on 5 vertices. The graph $\g=\overline{K}_4+K_{14}+W_5$ with boundary $\B=\overline{K}_4$ satisfies the hypothesis of Theorem \ref{thm:eg} and does so with the minimum possible number of vertices. In particular we have $|E| = 245$, $\chi(\g,\B)=4$, and $|V|=23$.
\end{eg}

The proof of Theorem \ref{thm:eg} can be found at the end of the section. First we need some preliminary results on the classes of arrangements (ii)--(iv).

\subsection{Ideal arrangements}

Let $\Phi\subseteq \KK^n$ be a finite root system with set of positive roots $\Phi^+$. A standard reference for root systems is \cite{humphreys1972}. The \emph{Coxeter arrangement} associated to $\Phi$ is the set of normal hyperplanes of $\Phi^+$. Every Coxeter arrangement associated to a classical root system $\mathsf{A}_n$, $\mathsf{B}_n$, $\mathsf{C}_n$ or $\mathsf{D}_n$ is a subset of an arrangement of the following type.

\begin{mydef}
	For all $n\geq 2$ let $\BB_n$ be the arrangement in $\KK^n$ of hyperplanes
		\[\{x_i = x_j : 1\leq i <j\leq n\}\cup\{x_i + x_j = 0 : 1\leq i<j\leq n\} \cup \{x_i = 0 : 1\leq i\leq n\}.\]
	\label{def:abd}
\end{mydef}

\begin{prop}
	If $\chi(\g,\B)\geq4$ and $|E|\geq 240$, then $\AA(\g,\u)$ is not combinatorially equivalent to any subarrangement of any Coxeter arrangement.
	
	\begin{proof}
		The matroids $M(\BB_n)$ are representable over any field $|\KK|$ with $|\KK|\geq 3$. However $M(\AA(\g,\u))$ is not representable over $\KK$ if $|\KK|<\chi(\g,\B)$ by \cite[Theorem 1.1(ii)]{lutz2018mat}. Hence if $\chi(\g,\B)\geq4$, then $\AA(\g,\u)$ is not combinatorially equivalent to any subarrangement of $\BB_n$.
		
		The exceptional root systems $\mathsf{E}_6$, $\mathsf{E}_7$, $\mathsf{E}_8$, $\mathsf{F}_4$ and $\mathsf{G}_2$ all have 240 or fewer elements. Hence no subarrangement of the associated Coxeter arrangements can have more than 240 elements. The result now follows from the classification of finite root systems.
	\end{proof}
	\label{prop:root}
\end{prop}

An \emph{ideal arrangement} (or a \emph{root ideal arrangement}) is a certain subarrangement of a Coxeter arrangement (see \cite{abe2014, hultman2016}). Graphic arrangements are subarrangements of $\BB_n$. Thus we have the following.

\begin{cor}
	If $\chi(\g,\B)\geq 4$ and $|E|\geq 240$, then $\AA(\g,\u)$ is not combinatorially equivalent to any ideal arrangement or graphic arrangement.
	\label{cor:root}
\end{cor}

\subsection{Hypersolvable arrangements}

Let $\AA$ be a central arrangement, and let $X\subseteq Y\subseteq \AA$. The containment $X\subseteq Y$ is \emph{closed} if $X\neq Y$ and $\{a,b,c\}$ is independent for all distinct $a,b\in X$ and $c\in Y\setminus X$. The containment $X\subseteq Y$ is \emph{complete} if $X\neq Y$ and for any distinct $a,b\in Y\setminus X$ there is $\gamma\in X$ such that $\{a,b,\gamma\}$ is dependent.

If $X\subseteq Y$ is closed and complete, then the element $\gamma$ is uniquely determined by $a$ and $b$. Write $\gamma = f(a,b)$. The containment $X\subseteq Y$ is \emph{solvable} if it is closed and complete, and if for any distinct $a,b,c\in Y\setminus X$ with $f(a,b)$, $f(a,c)$ and $f(b,c)$ distinct, the set $\{f(a,b),f(a,c),f(b,c)\}$ is dependent.

An increasing sequence $X_1 \subseteq\cdots\subseteq  X_k = \AA$ is called a \emph{hypersolvable composition series} for $\AA$ if $|X_1|=1$ and each $X_i \subseteq X_{i+1}$ is solvable.

\begin{mydef}[{\cite[Definition 1.8]{jambu1998}}]
	The central arrangement $\AA$ is \emph{hypersolvable} if it admits a hypersolvable composition series.
	\label{def:hyp}
\end{mydef}

 There is an analog for graphs. Let $S\subseteq T\subseteq E$. We say that $S\subseteq T$ is \emph{solvable} if it satisfies the following conditions:
\begin{enumerate}[(a)]
	\item There is no 3-cycle in $\g$ with two edges from $S$ and one edge from $T\setminus S$
	\item Either $T\setminus S = e$ with neither endpoint of $e$ met by $S$, or there exist distinct vertices $v_1,\ldots,v_k,v$ met by $T$ with $v_1,\ldots,v_k$ met by $S$ such that
	\begin{enumerate}[(i)]
		\item $S$ contains a clique on $\{v_1,\ldots,v_k\}$, and
		\item $T\setminus S = \{vv_s \in E : s=1,\ldots,k\}$.
	\end{enumerate}
\end{enumerate}
An increasing sequence $S_1\subseteq \cdots \subseteq S_k = E$ is called a \emph{hypersolvable composition series} for $\g$ if $|S_1|=1$ and each $S_i\subseteq S_{i+1}$ is solvable

\begin{mydef}[{\cite[Definition 6.6]{papadima2002}}]
	The graph $\g$ is \emph{hypersolvable} if it admits a hypersolvable composition series.
	\label{mydef:hypgraph}
\end{mydef}

\begin{prop}
	If the graph $\g$ is hypersolvable, then so is any induced subgraph of $\g$.
	
	\begin{proof}
		Suppose that $S_1\subseteq \cdots \subseteq S_k$ is a hypersolvable composition series for $\g$, and let $\overline{\g}$ be an induced subgraph of $\g$ with edge set $\overline{E}\subseteq E$. By eliminating empty sets and trivial containments in the sequence $S_1\cap \overline{E}\subseteq \cdots \subseteq S_k\cap \overline{E}$ one obtains a hypersolvable composition series for $\overline{\g}$.
	\end{proof}
	\label{prop:induced}
\end{prop}

The following proposition generalizes half a result of Papadima and Suciu \cite[Proposition 6.7]{papadima2002}, who showed that $\g$ is hypersolvable if and only if the associated graphic arrangement is hypersolvable.

\begin{prop}
	If $\AA(\g,\u)$ is hypersolvable, then the graph $\mg$, obtained from $\g$ by adding edges between every pair of boundary nodes, is hypersolvable.
	
	\begin{proof}
		Let $\me$ be the set of added edges, so that the edge set of $\mg$ is the disjoint union $E\cup \me$. Write $\B=\{v_1,\ldots,v_m\}$. For $i=1,\ldots,m-1$
		\[T_i = \{v_rv_s\in \me : r<s\leq i+1\},\]
		so for example $T_{m-1} = \me$.
		
		Suppose that $X_1\subseteq \cdots \subseteq X_k$ is a hypersolvable composition series for $\AA(\g,\u)$. For each $i$ let $S_i\subseteq \eo$ be the set corresponding to $X_i$. Let $j$ be the smallest index for which $\eh \in S_j$. Consider the increasing sequence
		\[
		S_1\subseteq \cdots \subseteq S_{j-1} \subseteq S_{j-1}\cup T_1 \subseteq \cdots \subseteq S_{j-1}\cup T_{m-1}\subseteq S_{j+1}\cup \me \subseteq \cdots \subseteq S_k\cup \me,
		\]
		omitting the initial portion $S_1\subseteq \cdots \subseteq S_{j-1}$ if $j=1$. It is routine to show that this sequence is a hypersolvable composition series for $\mg$.
	\end{proof}
\label{prop:hypersolvable}
\end{prop}

\begin{eg}
	Consider the network $N$ on the left side of Figure \ref{fig:nothyp}. Here $\mg=W_5$ is the wheel graph on 5 vertices. An exhaustive argument shows that $W_5$ is not hypersolvable. Hence Proposition \ref{prop:hypersolvable} implies that $\AA(\g,\u)$ is not hypersolvable.
	\label{eg:w5}
\end{eg}

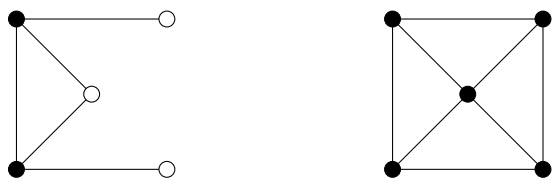
\begin{figure}[ht]
		\begin{tikzpicture}
	\def\c{1.4}
	\coordinate (a1) at (-1,-1);
	\coordinate (a2) at (1,-1);
	\coordinate (a3) at (1,1);
	\coordinate (a4) at (-1,1);
	\coordinate (a5) at (0,0);
	\draw (a1) -- (a2) -- (a3) -- (a4) -- (a1);
	\draw (a1) -- (a3);
	\draw (a2) -- (a4);
	\draw[fill=black] (a1) circle (3pt);
	\draw[fill=black] (a2) circle (3pt);
	\draw[fill=black] (a3) circle (3pt);
	\draw[fill=black] (a4) circle (3pt);
	\draw[fill=black] (a5) circle (3pt);
	
	\def\a{-5}
	\coordinate (b1) at (-1+\a,-1);
	\coordinate (b2) at (1+\a,-1);
	\coordinate (b3) at (1+\a,1);
	\coordinate (b4) at (-1+\a,1);
	\coordinate (b5) at (\a,0);
	\draw (b3) -- (b4) -- (b1) -- (b2);
	\draw (b1) -- (b5) -- (b4);
	\draw[fill=black] (b1) circle (3pt);
	\draw[fill=white] (b2) circle (3pt);
	\draw[fill=white] (b3) circle (3pt);
	\draw[fill=black] (b4) circle (3pt);
	\draw[fill=white] (b5) circle (3pt);
	\end{tikzpicture}
	\caption{Left-to-right: a network $N$ with boundary nodes marked in white and the associated graph $\mg=W_5$.}
	\label{fig:nothyp}
\end{figure}

\begin{q}
Does the converse of Proposition \ref{prop:hypersolvable} hold? In other words, is $\AA(\g,\u)$ hypersolvable whenever $\mg$ is hypersolvable?
\end{q}

\subsection{Disjoint broken circuits}

Fix an ordering of a central arrangement $\AA$, and let $\min X$ denote the minimal element of any $X\subseteq \AA$. The \emph{broken circuits} of $\AA$ are the sets $C\setminus \min C$ for all circuits $C$ of $\AA$. A broken circuit is \emph{minimal} if it does not properly contain any broken circuits. Van Le and R\"{o}mer \cite[Theorem 4.9]{vanle2013} answered Question \ref{q:koszul} affirmatively for all ordered arrangements with disjoint minimal broken circuits. No matter the ordering, many Dirichlet arrangements do not satisfy this requirement, as the following proposition implies.

\begin{prop}
	If there is an element of $V\setminus \B$ adjacent to at least 3 boundary nodes, then the minimal broken circuits of $\AA(\g,\u)$ are not disjoint with respect to any ordering.
	
	\begin{proof}
		Suppose that $i\in V\setminus \B$ is adjacent to distinct boundary nodes $j_1$, $j_2$ and $j_3$. Let $e_r$ be the edge $ij_r$ for $r=1,2,3$. Fix an ordering of $\AA(\g,\u)$ and suppose without loss of generality that $e_1<e_2<e_3$. We obtain circuits $\{e_0,e_1,e_3\}$ and $\{e_0,e_2,e_3\}$. The associated broken circuits are minimal, since there are no circuits of size $\leq 2$. Moreover both broken circuits contain $e_3$.
	\end{proof}
\label{prop:dmbc}
\end{prop}

\begin{proof}[Proof of Theorem \ref{thm:eg}]
	Since $\chi(\g,\B)\geq 4$ and $|E|\geq 240$, Corollary \ref{cor:root} says that $\AA(\g,\u)$ is not combinatorially equivalent to any ideal arrangement or graphic arrangement. Since $\g\setminus \B$ contains $W_5$ as an induced subgraph, $\mg$ also contains $W_5$ as an induced subgraph. Example \ref{eg:w5} and Propositions \ref{prop:induced} and \ref{prop:hypersolvable} imply that $\AA(\g,\u)$ is not hypersolvable, a property depending only on $M(\AA(\g,\u))$. Finally Proposition \ref{prop:dmbc} says that the broken circuits of $\AA(\g,\u)$ are not disjoint with respect to any ordering. This property only depends on $M(\AA(\g,\u))$, so the result follows.
\end{proof}

\section*{Acknowledgments}

\noindent The author thanks Trevor Hyde and the anonymous referee for helpful comments.


\bibliographystyle{abbrv}
\bibliography{lutz-kasfka}

\begin{thebibliography}{10}

\bibitem{abe2014}
T.~Abe, M.~Barakat, M.~Cuntz, T.~Hoge, and H.~Terao.
\newblock The freeness of ideal subarrangements of {W}eyl arrangements.
\newblock {\em Discrete Math. Theor. Comput. Sci.}, {DMTCS Proceedings vol. AT,
  26th International Conference on Formal Power Series and Algebraic
  Combinatorics (FPSAC)}, 2014.

\bibitem{beilinson1996}
A.~Beilinson, V.~Ginzburg, and W.~Soergel.
\newblock Koszul duality patterns in representation theory.
\newblock {\em J. Amer. Math. Soc.}, 9(2):473--527, 1996.

\bibitem{froberg1999}
R.~Fr\"{o}berg.
\newblock Koszul algebras.
\newblock In D.~E. Dobbs, M.~Fontana, and S.-E. Kabbaj, editors, {\em Advances
  in Commutative Ring Theory}, Lecture Notes in Pure and Applied Mathematics.
  Marcel Dekker, 1999.

\bibitem{hultman2016}
A.~Hultman.
\newblock Supersolvability and the {K}oszul property of root ideal
  arrangements.
\newblock {\em Proc. Amer. Math. Soc.}, 144(4):1401--1413, 2016.

\bibitem{humphreys1972}
J.~E. Humphreys.
\newblock {\em Introduction to Lie Algebras and Representation Theory}.
\newblock Graduate Texts in Mathematics. Springer, 1972.

\bibitem{jambu1998}
M.~Jambu and S.~Papadima.
\newblock A generalization of fiber-type arrangements and a new deformation
  method.
\newblock {\em Topology}, 37(6):1135--1164, 1998.

\bibitem{lutz2017}
B.~Lutz.
\newblock Electrical networks and hyperplane arrangements.
\newblock {\em (preprint)}, 2017.
\newblock \href{arxiv.org/abs/1709.01227}{\texttt{arXiv:1709.01227 [math.CO]}}.

\bibitem{lutz2018mat}
B.~Lutz.
\newblock Electrical networks and frame matroids.
\newblock {\em (preprint)}, 2018.
\newblock \href{arxiv.org/abs/1809.10100}{\texttt{arXiv:1809.10100[math.CO]}}.

\bibitem{mu2015}
L.~Mu and R.~P. Stanley.
\newblock Supersolvability and freeness for {$\psi$}-graphical arrangements.
\newblock {\em Discrete Comput. Geom.}, 53(4):965--970, 2015.

\bibitem{papadima2002}
S.~Papadima and A.~I. Suciu.
\newblock Higher homotopy groups of complements of complex hyperplane
  arrangements.
\newblock {\em Adv. Math.}, 165(1):71--100, 2002.

\bibitem{peeva2010}
I.~Peeva.
\newblock {\em Graded Syzygies}, volume~14 of {\em Algebra and Applications}.
\newblock Springer, 2010.

\bibitem{schenck2002}
H.~Schenck and A.~Suciu.
\newblock Lower central series and free resolutions of hyperplane arrangements.
\newblock {\em Trans. Amer. Math. Soc.}, 354(9):3409--3433, 2002.

\bibitem{shelton1997}
B.~Shelton and S.~Yuzvinsky.
\newblock Koszul algebras from graphs and hyperplane arrangements.
\newblock {\em J. Lond. Math. Soc.}, 56(3):477--490, 1997.

\bibitem{stanley1972}
R.~P. Stanley.
\newblock Supersolvable lattices.
\newblock {\em Algebra Universalis}, 2(1):197--217, 1972.

\bibitem{stanley2007}
R.~P. Stanley.
\newblock An introduction to hyperplane arrangements.
\newblock In E.~Miller, V.~Reiners, and B.~Sturmfels, editors, {\em Geometric
  Combinatorics}, volume~13 of {\em IAS/Park City Math. Ser.}, pages 389--496.
  AMS, 2007.

\bibitem{stanley2015}
R.~P. Stanley.
\newblock Valid orderings of real hyperplane arrangements.
\newblock {\em Discrete Comput.~Geom.}, 53(4):951--964, 2015.

\bibitem{suyama2018}
D.~Suyama and S.~Tsujie.
\newblock Vertex-weighted graphs and freeness of {$\psi$}-graphical
  arrangements.
\newblock {\em Discrete Comput. Geom.}, 2018.

\bibitem{vanle2013}
D.~Van~Le and T.~R{\"o}mer.
\newblock Broken circuit complexes and hyperplane arrangements.
\newblock {\em J. Algebraic Combin.}, 38(4):989--1016, 2013.

\end{thebibliography}

\end{document}